\documentclass{article}

\usepackage{graphicx}
\usepackage{color}
\usepackage{amsmath,amssymb,amsthm,amsrefs,amsfonts,mathrsfs,indentfirst}
\usepackage[all]{xy}
\usepackage{mathtools}
\usepackage{hyperref}

\DeclareMathAlphabet{\mathsfsl}{OT1}{cmss}{m}{sl}

\newtheorem{thm}{Theorem}[section]
\newtheorem{lem}[thm]{Lemma}
\newtheorem{cor}[thm]{Corollary}
\newtheorem{prop}[thm]{Proposition}

\theoremstyle{definition}

\begin{document}

\title{Null-homotopic knots have Property R}

\author{{\Large Yi NI}\\{\normalsize Department of Mathematics, Caltech, MC 253-37}\\
{\normalsize 1200 E California Blvd, Pasadena, CA
91125}\\{\small\it Email\/:\quad\rm yini@caltech.edu}}

\date{}

\maketitle

\begin{abstract}
We prove that if $K$ is a nontrivial null-homotopic knot in a closed oriented $3$--manfiold $Y$ such that $Y-K$ does not have an $S^1\times S^2$ summand, then the zero surgery on $K$ does not have an $S^1\times S^2$ summand. This generalizes a result of Hom and Lidman, who proved the case when $Y$ is an irreducible rational homology sphere.
\end{abstract}

\

\section{Introduction}

Given a null-homologous knot $K$ in a $3$--manifold $Y$ and a slope $\frac pq\in\mathbb Q\cup\{\infty\}$, let $Y_{p/q}(K)$ be the $\frac pq$--surgery on $K$. Gabai's famous Property R Theorem \cite{G3} asserts, among others, that if $K$ is a nontrivial knot in $S^3$, then $S^3_0(K)$ is irreducible. In particular, $S^3_0(K)$ does not have an $S^1\times S^2$ summand. 

In recently years, many generalizations of this theorem have been proved using Heegaard Floer homology. 
See, for example, the overview in \cite{NiPropG}. Hom and Lidman \cite{HL} proved two generalizations of Property R. One result they proved is, if $K$ is a nontrivial null-homotopic knot in an irreducible rational homology sphere $Y$, then $Y_0(K)$ does not have an $S^1\times S^2$ summand. The goal of this note is to remove the restrictions on the ambient manifold.

\begin{thm}\label{thm:PropR}
Let $Y$ be a closed, oriented, connected $3$--manifold, and $K\subset Y$ be a nontrivial null-homotopic knot such that $Y-K$ does not have an $S^1\times S^2$ summand, then $Y_0(K)$ does not have an $S^1\times S^2$ summand.
\end{thm}

In Gabai's work \cite{G3}, it is proved that $S^3_0(K)$ remembers the information of $K$ about the genus and fiberedness. Motivated by this result, a concept ``Property G'' was introduced in \cite{NiNSSphere} as a generalization of Property R. Known results on Property G are summarized in \cite{NiPropG}. We will not give the complete definition of Property G here. Instead, we just state the explicit result for genus--$1$ null-homotopic knot.

\begin{cor}\label{cor:Genus1}
Let $K\subset Y$ be a genus--$1$ null-homotopic knot, then $K$ has Property G. That is, if $F$ is a genus--$1$ Seifert surface bounded by $K$, and $\widehat{F}\subset Y_0(K)$ is the torus obtained by capping off $\partial F$ with a disk, then $[\widehat F]\in H_2(Y_0(K))$ is not represented by a sphere. Moreover, if $Y_0(K)$ is a torus bundle over $S^1$ with fiber $\widehat{F}$, then $K$ is a fibered knot with fiber $F$.
\end{cor}

Corollary~\ref{cor:Genus1} answers the genus--$1$ case of a question of Boileau 
\cite[Problem~1.80C]{KirbyList}. There are easy counterexamples to the original question of Boileau, so one should modify the question to add the condition on the fiber of the zero surgery. See \cite{NiFibredSurg} for more details.

The strategy of the proof of Theorem~\ref{thm:PropR} is as follows. If $b_1(Y)>0$, the theorem easily follows from a result of Lackenby \cite{Lackenby} and Gabai \cite{G2}. If $b_1(Y)=0$, we use results about degree-one maps and a result in \cite{DLVVW} to show that if $Y_0(K)=Z\#(S^1\times S^2)$ then $\pi_1(Z)\cong\pi_1(Y)$. Theorem~\ref{thm:PropR} then follows from work of Hom and Lidman \cite{HL}.

We will use the following notations. If $N$ is a submanifold of a manifold $M$, let $\nu(N)$ be a closed tubular neighborhood of $N$, and let $\nu^{\circ}(N)$ be the interior of $\nu(N)$. If $X,Y$ are two spaces, $f:X\to Y$ is a continuous map, let $f_*:\pi_1(X)\to \pi_1(Y)$ be the induced map. We will always suppress the base point in the notation when we talk about fundamental groups. 

This paper is organized as follows. In Section~\ref{sect:DegreeOne}, we prove general results about degree-one maps with certain properties on the induced homomorphisms on $\pi_1$. In Section~\ref{sect:ZeroSurg}, we prove that if the zero surgery on a knot in $Y$ is $Z\#(S^1\times S^2)$, then $\pi_1(Z)\cong\pi_1(Y)$. In Section~\ref{sect:Main}, we use work of Lackenby \cite{Lackenby} and Hom--Lidman \cite{HL} to prove Theorem~\ref{thm:PropR}. Corollary~\ref{cor:Genus1} is also proved as an application of this theorem.

\

\noindent{\bf Acknowledgements.} This
research was funded by NSF grant numbers DMS-1252992 and DMS-1811900. We are graterful to Tye Lidman for very helpful comments on an earlier draft of this paper.


\section{Degree-one maps which induce surface-group injective homomorphisms}\label{sect:DegreeOne}

In this section, we will prove results about degree-one maps which induce surface-group injective homomorphisms on $\pi_1$.

A group $\Gamma$ is a {\it surface group} if it is isomorphic to the fundamental group of a closed orientable surface.
Let $\varphi: G\to H$ be a group homomorphism. We say $\varphi$ is {\it surface-group injective}, if the restriction of $\varphi$ to every surface subgroup of $G$ is injective. 

\begin{lem}\label{lem:SurfInj}
Let $\varphi: G_1*G_2\to H$ be a group homomorphism.
If both $\varphi|_{G_1}$ and $\varphi|_{G_2}$ are surface-group injective, then $\varphi$ is also surface-group injective.
\end{lem}
\begin{proof}
Let $\Gamma$ be a surface subgroup of $G_1*G_2$.
By the Kurosh Subgroup Theorem \cite[Theorem~8.3]{Hempel}, $\Gamma$ is the free product of a free group and conjugates of subgroups of $G_i$, $i=1,2$. Since $\Gamma$ is not a nontrivial free product, it is conjugated to a subgroup of $G_i$ for some $i$. Since $\varphi|_{G_i}$ is surface-group injective, $\varphi$ is also injective on $\Gamma$.
\end{proof}

The importance of the concept of surface-group injective maps is illustrated by the next lemma.

\begin{lem}\label{lem:Compress}
Let $X,Y$ be closed, oriented, connected $3$--manifolds, $f:X\to Y$ be a surjective map such that $f_*$ is surface-group injective. Let $S\subset Y$ be a separating $2$--sphere, and assume that $R=f^{-1}(S)$ is a closed, oriented, connected surface. Then there exists a separating $2$--sphere $E\subset X$ so that $R$ is obtained by adding tubes to $E$.
\end{lem}
\begin{proof}
Let $\iota: R\to X$ be the inclusion map. Since $f(R)=S$ is a sphere, \[\iota_*(\pi_1(R))\subset\ker f_*.\]
If $R$ is not a sphere, since $f_*$ is surface-group injective, $R$ must be compressible. Let $R'$ be the surface obtained by compressing $R$, then $R$ can be obtained from $R'$ by adding a tube. Let $R_1'$ be a component of $R'$,
let $\iota': R_1'\to X$ be the inclusion map, then 
\[
\iota'_*(\pi_1(R_1'))=\iota_*(\pi_1(R_1'\cap R))\subset\iota_*(\pi_1(R))\subset\ker f_*.
\]
If $R'$ does not consist of spheres, since $f_*$ is surface-group injective, $R'$ must be compressible. So $R'$ can be obtained from another surface $R''$ by adding a tube. 

Continue with the above process, we conclude that $R$ can be obtained from some spheres by adding tubes. We can rearrange the order of the tubes, so that some tubes connecting different spheres are added first to get a single sphere $E$, then $R$ is obtained by adding other tubes to $E$. 

Let $y_1,y_2\in Y$ be two points separated by $S$. Since $f$ is surjective, both $f^{-1}(y_1)$ and $f^{-1}(y_2)$ are non-empty. These two sets are clearly separated by $R$, so $R$ is separating. The process of compressing a surface does not change the homology class of the surface, hence
$E$ is also separating.
\end{proof}

In the rest of this section, let $Y$ be a $3$--manifold which has no $S^1\times S^2$ summand, $S_1,S_2,\dots,S_n$ be a collection of disjoint $2$--spheres in $Y$ satisfying the following conditions: $Y\setminus (\cup_{i=1}^n S_i)$ has $n+1$ components whose closures are $\check Y_1,\check Y_2,\dots,\check Y_n,\check Y_{n+1}$, where $\check Y_{n+1}$ is $S^3$ with $n$ open balls removed, and a closed irreducible manifold $Y_i\ne S^3$ can be obtained from $\check Y_i$ by capping off $\partial\check Y_i=S_i$ with a ball $B_i$, $1\le i\le n$. Then
\[
Y=\#_{i=1}^nY_i.
\]
When $Y$ is irreducible, it is understood that $n=0$.

\begin{prop}\label{prop:gMap}
Let $X$ be a closed, oriented, connected $3$--manifold, and $f:X\to Y$ be a degree-one map such that $f_*$ is surface-group injective. Then there exists a degree-one map $g: X\to Y$ satisfying $g_*=f_*$, and each $E_i=g^{-1}(S_i)$ is a $2$--sphere.
\end{prop}
\begin{proof}
We induct on $n$. When $n=0$, there is nothing to prove. So we assume $n>0$ and the result is proved for $n-1$.

Using \cite[Theorem~1.1]{RongWang}, we may assume $R_1=f^{-1}(S_1)$ is a connected surface. Let $\check Y_0=\overline{Y\setminus\check Y_1}$, and let $Y_0$ be obtained by capping off $\partial \check Y_0$ with a ball $B_0$. Let $U_i=f^{-1}(\check Y_i)$, $i=0,1$.
Then $\partial U_1=\partial U_0=R_1$. 

By Lemma~\ref{lem:Compress}, there exists a separating $2$--sphere $E_1\subset X$, so that $R_1$ is obtained by adding tubes to $E_1$. Now $E_1$ splits $X$ into two parts $\check X_1,\check X_0$, so that $\check X_i$ can be obtained from $U_i$ by adding $1$--handles and digging tunnels, $i=0,1$. Let $X_i$ be the closed manifold obtained by capping off $\partial\check X_i$ with a ball, $i=0,1$. 

We claim that each map $f|_{U_i}:U_i\to \check Y_i$ can be extended to a degree-one map $f_i: X_i\to Y_i$. In fact, the manifold $X_i$ can be obtained from $U_i$ by gluing a $3$--manifold $V_i$ which is obtained from $B^3$ by digging tunnels and adding $1$--handles. Since $B_i$ is a ball which is contractible, we can extend $f|_{U_i}:U_i\to \check Y_i$ to a map $f_i: X_i\to Y_i$ by sending $V_i$ to $B_i$. The degree of $f_i$ is $1$ since the degree of $f|_{U_i}$ is $1$.

Since $\deg f_i=1$, after a homotopy supported in $V_i$, we may assume there exists a ball $B'_i\subset\mathrm{int}(B_i)$, such that $f_i$ sends $B^{\star}_i=f_i^{-1}(B'_i)$ homeomorphically onto $B_i'$. Now we can glue $X_i\setminus\mathrm{int}(B_i^{\star})$, $i=0,1$, together along their boundary, to get back \[X=(X_0\setminus\mathrm{int}(B_0^{\star}))\cup_{S^2} (X_1\setminus\mathrm{int}(B_1^{\star})),\] and define a map 
\[f_0\#f_1:X\to Y=(Y_0\setminus\mathrm{int}(B_0'))\cup_{S^2} (Y_1\setminus\mathrm{int}(B_1'))\] by gluing the restrictions of $f_0,f_1$. We rename \[S_1=\partial (Y_0\setminus\mathrm{int}(B_0')), E_1=\partial (X_0\setminus\mathrm{int}(B_0^{\star})),\] then $E_1=(f_0\#f_1)^{-1}(S_1)$.

We claim that $(f_0\#f_1)_*=f_*$. Let $D\subset E_1$ be a disk such that all tubes in $S_1$ are added to the interior of $D$, and let $D^c=\overline{E_1\setminus D}$. Let $V\subset X$ be the handlebody obtained by adding the $1$--handles bounded by the tubes to $\nu(D)$. Since $\pi_1(V)$ is a quotient of $\pi_1(\partial V)$, the map $X\setminus V\to X$ induces a surjective map on $\pi_1$. To prove $(f_0\#f_1)_*=f_*$, we only need to prove that $(f_0\#f_1)_*(\alpha)=f_*(\alpha)$ when $\alpha$ is a homotopy class represented by a loop in $X\setminus V$. We observe that $f_0\#f_1=f$ on $X\setminus(V\cup\nu(D^c))$, and $\pi_1(X\setminus V)$ is the free product of the $\pi_1$ of the two components of $X\setminus(V\cup\nu(D^c))$, so $(f_0\#f_1)_*(\alpha)=f_*(\alpha)$.

By the induction hypothesis, there exists a map $g_1:X_1\to Y_1$, such that $g_1^{-1}(S_i)$ is a sphere for $2\le i\le n$, and $(g_1)_*=(f_1)_*$. We can define a map $g=f_0\#g_1$ in a similar way as $f_0\#f_1$, then $g^{-1}(S_i)$ is a sphere for $1\le i\le n$, and
\[
g_*=(f_0\#g_1)_*=(f_0\#f_1)_*=f_*.
\]
This finishes the induction step.
\end{proof}


\section{Zero surgery on a null-homotopic knot}\label{sect:ZeroSurg}

The goal of this section is to prove the following proposition.

\begin{prop}\label{prop:Pi1Isom}
Let $Y$ be a closed, connected, oriented $3$--manifold which does not have an $S^1\times S^2$ summand, and $K\subset Y$ be a null-homotopic knot. If $Y_0(K)=Z\#(S^1\times S^2)$, then $\pi_1(Z)\cong\pi_1(Y)$.
\end{prop}

We first prove a general result about surgery on null-homotopic knots.

\begin{lem}\label{lem:Retraction}
Let $Y$ be a closed, oriented, connected $3$--manifold, and $K\subset Y$ be a null-homotopic knot. 
Let $V$ be the $2$--handle cobordism from $Y$ to $Y_m(K)$ for some integer $m$. Then there exists a retraction $p:V\to Y$, so that $p|_{Y_m(K)}:Y_m(K)\to Y$ is a degree-one map.
\end{lem}
\begin{proof}
The cobordism $V$ deformation retracts to the space $V'$ obtained from $Y$ by adding a $2$--cell $e^2$ along $K$. Since $K$ is null-homotopic, the identity map on $Y$ can be extended over $e^2$. Hence we have a retraction $V'\to Y$, which implies the existence of the retraction $p:V\to Y$.
The degree of the restriction $p|_{Y_m(K)}:Y_m(K)\to Y$ is $1$ since it induces an isomorphism on $H_3$. 
\end{proof}

The existence of the above degree-one map is a well-known result. See \cite[Proposition~3.2]{BW} and \cite{Gadgil}. 

In the rest of this section, let $V$ be the $2$--handle cobordism from $Y$ to $X=Y_0(K)$, $p:V\to Y$
be the retraction in Lemma~\ref{lem:Retraction}, and $f=p|_{Y_0(K)}$. Moreover, we assume $Y_0(K)=Z\#(S^1\times S^2)$. Let $\check Z$ be the submanifold of $Y_0(K)$ which is $Z$ with a ball removed. Since $Y_0(K)=Z\#(S^1\times S^2)$, we can add a $3$--handle to $V$ to get a cobordism $W:Y\to Z$. 

\begin{lem}\label{lem:ZInj}
The restriction of $f_*$ to $\pi_1(\check Z)$ is injective.
\end{lem}
\begin{proof}
Consider the following commutative diagram, where all maps except $p$ are inclusions:
\[
\xymatrix{
Y\ar@/^/[rd]^{\iota_Y} &Y_0(K)\ar[d]^{\iota_0} &\check Z\ar[l]\ar[r]\ar[ld]^{\iota_{\check Z}} &Z\ar[ld]^{\iota_Z}\\
&V\ar@/^/[lu]^{p}\ar[r] &W&
}.
\]
Then $f=p\circ\iota_0$.

Since the $2$--handle in $V$ is added along the null-homotopic knot $K$, the inclusion $\iota_Y:Y\to V$ induces an isomorphism on $\pi_1$. Since $p\circ\iota_Y=\mathrm{id}_Y$, $p_*$ is an isomorphism.

Since $W$ is obtained from $V$ by adding a $3$--handle, the inclusion
$V\subset W$ induces an isomorphism on $\pi_1$.
We have the commutative diagram
\[
\xymatrix{
 &\pi_1(\check Z)\ar[r]^{\cong}\ar[ld]_{(\iota_{\check Z})_*} &\pi_1(Z)\ar[ld]^{(\iota_Z)_*}\\
\pi_1(V)\ar[r]^{\cong} &\pi_1(W)&
}.
\]
The manifold $W$ (after being turned up-side-down) can be obtained from $Z\times I$ by adding a $1$--handle and a $2$--handle, and the $2$--handle cobordism is exactly $V$ being turned up-side-down. 
By \cite[Proposition~2.1]{DLVVW}, $(\iota_Z)_*$ is injective, so $(\iota_{\check Z})_*$ is also injective.

Now consider the commutative diagram
\[
\xymatrix{
\pi_1(Y) &\pi_1(Y_0(K))\ar[d]^{(\iota_0)_*} &\pi_1(\check Z)\ar[l]\ar[ld]^{(\iota_{\check Z})_*} \\
&\pi_1(V)\ar[lu]^{p_*} &
}.
\]
The restriction of $f_*$ to $\pi_1(\check Z)$ is just $p_*\circ(\iota_{\check Z})_*$, which is injective since $(\iota_{\check Z})_*$ is injective and $p_*$ is an isomorphism.
\end{proof}

\begin{cor}
The induced map $f_*:\pi_1(Y_0(K))\to \pi_1(Y)$ is surface-group injective.
\end{cor}
\begin{proof}
This follows from Lemmas~\ref{lem:SurfInj} and~\ref{lem:ZInj}.
\end{proof}

\begin{proof}[Proof of Proposition~\ref{prop:Pi1Isom}]
We will use the notations in Section~\ref{sect:DegreeOne}. Since $f_*$ is surface-group injective, we can apply Proposition~\ref{prop:gMap} to get a degree-one map $g:X=Y_0(K)\to Y$ so that $g_*=f_*$ and $g^{-1}(S_i)=E_i$ is a separating sphere whenever $1\le i\le n$. 

Since $X$ has an $S^1\times S^2$ summand, by the uniqueness part of the Kneser--Milnor theorem, one component of $X\setminus(\cup_{i=1}^n E_i)$ has an $S^1\times S^2$ summand. 

If the $S^1\times S^2$ summand is in $g^{-1}(\check Y_{n+1})$, then $g_*(S^1\times\{point\})$ is null-homotopic. It follows that 
\[g_*(\pi_1(X)=g_*(\pi_1(Z\#(S^1\times S^2)))=g_*(\pi_1(Z)).\]
Since $\deg g=1$, $g_*$ is surjective. So $g_*|_{\pi_1(Z)}$ is surjective. Our result follows from Lemma~\ref{lem:ZInj} since $g_*=f_*$.

If $\check X_i=g^{-1}(\check Y_i)$ has an $S^1\times S^2$ summand for some $i$ satisfying $1\le i\le n$, without loss of generality, we may assume $i=1$. The map $g|_{\check X_1}$ extends to a map $g_1:X_1\to Y_1$.
Suppose that $X_1=Z_1\#(S^1\times S^2)$, let $P\subset X_1$ be $\{\mathrm{point}\}\times S^2$. Then $X_{1}\setminus \nu^{\circ}(P)$ is homeomorphic to $Z_1$ with two open balls removed. Since  $\pi_2(Y_1)=0$, $(g_1)|_P$ is null-homotopic in $Y_1$. We can then extend  $g_1|_{X_1\setminus \nu^{\circ}(P)}$ to a map $h_1: Z_1\to Y_1$. The new map $h_1$ is again a degree-one map, so $(h_1)_*$ is surjective. 

Using Lemma~\ref{lem:ZInj}, we see that $g_*$ is injective on $\pi_1(X_0\#Z_1)$. (Recall that $X_0$ is obtained from $X$ by replacing $\check X_1$ with a ball.) In particular, $(h_1)_*$ is injective, so \[\pi_1(Z_1)\cong\pi_1(Y_1).\]
We also get that $g_*$ is injective on $\pi_1(X_0)$.  Since $g|_{\check X_0}:\check X_0\to\check Y_0$ is a degree-one proper map, $(g|_{\check X_0})*$ is surjective. So
\[
\pi_1(X_0)\cong \pi_1(Y_0).
\]

Since $Z\cong X_0\#Z_1$, $Y=Y_0\#Y_1$,
we have
\[
\pi_1(Z)\cong\pi_1(X_0)*\pi_1(Z_1)\cong\pi_1(Y_0)*\pi_1(Y_1)\cong\pi_1(Y).\qedhere
\]
\end{proof}


\section{Proof of the main theorem}\label{sect:Main}

In this section, we will prove Theorem~\ref{thm:PropR} and Corollary~\ref{cor:Genus1}.

\begin{proof}[Proof of Theorem~\ref{thm:PropR} when $b_1(Y)>0$]
Without loss of generality, we may assume $M=Y\setminus\nu^{\circ}(K)$ is irreducible. Since $b_1(Y)>0$, there exists a closed, oriented, connected surface $S$ in the interior of $M$, so that $S$ is taut in $M$. Notice that for the $\infty$ slope on $K$, the core of the surgery solid torus, which is $K$, is null-homotopic. Using \cite[Theorem~A.21]{Lackenby}, which is a stronger version of the main result in \cite{G2}, we conclude that each $2$--sphere in $Y_0(K)$ bounds a rational homology ball. Hence $Y_0(K)$ does not have an $S^1\times S^2$ summand.
\end{proof}

\begin{prop}\label{prop:SameRank}
Let $Y_1,Y_2$ be two closed, oriented, connected $3$--manifolds. If $\pi_1(Y_1)\cong\pi_2(Y_2)$, then \[\mathrm{rank}\widehat{HF}(Y_1)=\mathrm{rank}\widehat{HF}(Y_2).\]
\end{prop}
\begin{proof}
This is a well-known consequence of the Geometrization Theorem. As in \cite[Theorem~2.1.3]{AFW}, if $\pi_1(Y_1)\cong\pi_2(Y_2)$, then there is a one-to-one correspondence between the summands of $Y_1$ and the summands of $Y_2$, such that any pair of summands in the correspondence consists of either homeomorphic manifolds, (the homemorphism does not necessarily preserve the orientation), or lens spaces with the same $H_1$. Our result then follows from basic properties of Heegaard Floer homology \cite{OSzAnn1,OSzAnn2}.
\end{proof}

\begin{proof}[Proof of Theorem~\ref{thm:PropR} when $b_1(Y)=0$]
Without loss of generality, we may assume $Y\setminus K$ is irreducible. 
Assume that $Y_0(K)=Z\#(S^1\times S^2)$. By Proposition~\ref{prop:Pi1Isom}, $\pi_1(Y)\cong\pi_1(Z)$. Using Proposition~\ref{prop:SameRank}, we get $\mathrm{rank}\widehat{HF}(Y)=\mathrm{rank}\widehat{HF}(Z)$. Our conclusion follows from \cite[Theorem~1.1]{HL}.
\end{proof}

\begin{proof}[Proof of Corollary~\ref{cor:Genus1}]
The first statement follows from Theorem~\ref{thm:PropR}. We only need to prove the second statement. In this case, $Y_0(K)$ is a torus bundle over $S^1$. By Lemma~\ref{lem:Retraction}, there exists a degree-one map $f:Y_0(K)\to Y$. Then $f_*$ is surjective.
We claim that $Y$ is $S^1\times S^2,\mathbb RP^3\#\mathbb RP^3$ or a spherical manifold. Then our conclusion follows from \cite{AiNi}.

If $Y$ is reducible, then either $Y=S^1\times S^2$ or $Y$ is a nontrivial connected sum. If $Y=S^1\times S^2$, we are done. Now we assume $Y$ is a nontrivial connected sum, so $\pi_1(Y)\cong A*B$ with $A,B$ nontrivial.
Let $T$ be a fiber of $Y_0(K)$, then $f_*(\pi_1(T))$ is an abelian normal subgroup of $\pi_1(Y)=f_*(\pi_1(Y_0(K)))$.
By the Kurosh Subgroup Theorem \cite[Theorem~8.3]{Hempel}, $f_*(\pi_1(T))$ is also a free product of free groups and conjugates of subgroups of $A,B$. Since $f_*(\pi_1(T))$ is abelian, it must be either a subgroup of $\mathbb Z$ or the conjugate of a subgroup of $A$ or $B$. Since $f_*(\pi_1(T))$ is a normal subgroup of $\pi_1(Y)\cong A*B$, the latter case cannot happen. 

Now $f_*(\pi_1(T))$ is a subgroup of $\mathbb Z$. Since $f_*$ is surjective, $\pi_1(Y)/f_*(\pi_1(T))$ is a cyclic group.
If $f_*(\pi_1(T))=\{1\}$, then $\pi_1(Y)$ is cyclic, a contradiction to $\pi_1(Y)\cong A*B$. So $f_*(\pi_1(T))\cong\mathbb Z$.

 If $\pi_1(Y)/f_*(\pi_1(T))\cong\mathbb Z$, then $\pi_1(Y)$ contains $\mathbb Z^2$ as a finite index subgroup, which is not possible since $\mathbb Z^2$ is not the fundamental group of any closed $3$--manifold. If  $\pi_1(Y)/f_*(\pi_1(T))$ is finite, then $\pi_1(Y)$ contains $\mathbb Z$ as a finite index subgroup. It follows that $Y$ is finitely covered by $S^1\times S^2$, thus it must be $S^1\times S^2$ or $\mathbb RP^3\#\mathbb RP^3$.

Now we consider the case $Y$ is irreducible. If $Y$ is a spherical manifold, we are done. If $Y$ is irreducible and not a spherical manifold, then $f_*$ is injective by \cite[Theorem~4]{Wang}. So $f_*$ is an isomorphism, a contradiction to the fact that $b_1(Y_0(K))>b_1(Y)$.
\end{proof}


\begin{thebibliography}{H}

\bib{AiNi}{article}{
   author={Ai, Yinghua},
   author={Ni, Yi},
   title={Two applications of twisted Floer homology},
   journal={Int. Math. Res. Not. IMRN},
   date={2009},
   number={19},
   pages={3726--3746},
   issn={1073-7928},
}

\bib{AFW}{book}{
   author={Aschenbrenner, Matthias},
   author={Friedl, Stefan},
   author={Wilton, Henry},
   title={3-manifold groups},
   series={EMS Series of Lectures in Mathematics},
   publisher={European Mathematical Society (EMS), Z\"{u}rich},
   date={2015},
   pages={xiv+215},
   isbn={978-3-03719-154-5},
}

\bib{BW}{article}{
   author={Boileau, Michel},
   author={Wang, Shicheng},
   title={Non-zero degree maps and surface bundles over $S^1$},
   journal={J. Differential Geom.},
   volume={43},
   date={1996},
   number={4},
   pages={789--806},
   issn={0022-040X},
}

\bib{DLVVW}{article}{
   author={Daemi, Aliakbar},
   author={Lidman, Tye},
   author={Vela-Vick, David Shea},
   author={Wong, C.-M. Michael}
   title={Ribbon homology cobordisms},
   status={preprint},
   date={2019},
   eprint={https://arxiv.org/abs/1904.09721}
}


\bib{G2}{article}{
   author={Gabai, David},
   title={Foliations and the topology of $3$-manifolds. II},
   journal={J. Differential Geom.},
   volume={26},
   date={1987},
   number={3},
   pages={461--478},
   issn={0022-040X},
}

\bib{G3}{article}{
   author={Gabai, David},
   title={Foliations and the topology of $3$-manifolds. III},
   journal={J. Differential Geom.},
   volume={26},
   date={1987},
   number={3},
   pages={479--536},
   issn={0022-040X},
}

\bib{Gadgil}{article}{
   author={Gadgil, Siddhartha},
   title={Degree-one maps, surgery and four-manifolds},
   journal={Bull. Lond. Math. Soc.},
   volume={39},
   date={2007},
   number={3},
   pages={419--424},
   issn={0024-6093},
}


\bib{Hempel}{book}{
   author={Hempel, John},
   title={3-manifolds},
   note={Reprint of the 1976 original},
   publisher={AMS Chelsea Publishing, Providence, RI},
   date={2004},
   pages={xii+195},
   isbn={0-8218-3695-1},
}

\bib{HL}{article}{
   author={Hom, Jennifer},
   author={Lidman, Tye},
   title={Dehn surgery and non-separating two-spheres},
   status={preprint},
   date={2020},
   eprint={https://arxiv.org/abs/2006.11249}
}

\bib{KirbyList}{article}{
   title={Problems in low-dimensional topology},
   editor={Kirby, Rob},
   conference={
      title={Geometric topology},
      address={Athens, GA},
      date={1993},
   },
   book={
      series={AMS/IP Stud. Adv. Math.},
      volume={2},
      publisher={Amer. Math. Soc., Providence, RI},
   },
   date={1997},
   pages={35--473},
}

\bib{Lackenby}{article}{
   author={Lackenby, Marc},
   title={Dehn surgery on knots in $3$-manifolds},
   journal={J. Amer. Math. Soc.},
   volume={10},
   date={1997},
   number={4},
   pages={835--864},
   issn={0894-0347},
}

\bib{NiFibredSurg}{article}{
   author={Ni, Yi},
   title={Dehn surgeries that yield fibred 3-manifolds},
   journal={Math. Ann.},
   volume={344},
   date={2009},
   number={4},
   pages={863--876},
   issn={0025-5831},
}


\bib{NiNSSphere}{article}{
   author={Ni, Yi},
   title={Nonseparating spheres and twisted Heegaard Floer homology},
   journal={Algebr. Geom. Topol.},
   volume={13},
   date={2013},
   number={2},
   pages={1143--1159},
   issn={1472-2747},
}

\bib{NiPropG}{article}{
 author={Ni, Yi},
   title={Property G and the $4$--genus},
   status={preprint},
   date={2020},
   eprint={https://arxiv.org/abs/2007.03721}
}

\bib{OSzAnn1}{article}{
   author={Ozsv\'{a}th, Peter},
   author={Szab\'{o}, Zolt\'{a}n},
   title={Holomorphic disks and topological invariants for closed
   three-manifolds},
   journal={Ann. of Math. (2)},
   volume={159},
   date={2004},
   number={3},
   pages={1027--1158},
   issn={0003-486X},
}

\bib{OSzAnn2}{article}{
   author={Ozsv\'{a}th, Peter},
   author={Szab\'{o}, Zolt\'{a}n},
   title={Holomorphic disks and three-manifold invariants: properties and
   applications},
   journal={Ann. of Math. (2)},
   volume={159},
   date={2004},
   number={3},
   pages={1159--1245},
   issn={0003-486X},
}	

\bib{RongWang}{article}{
   author={Rong, Yongwu},
   author={Wang, Shicheng},
   title={The preimages of submanifolds},
   journal={Math. Proc. Cambridge Philos. Soc.},
   volume={112},
   date={1992},
   number={2},
   pages={271--279}
}

\bib{Wang}{article}{
   author={Wang, Shicheng},
   title={The existence of maps of nonzero degree between aspherical
   $3$-manifolds},
   journal={Math. Z.},
   volume={208},
   date={1991},
   number={1},
   pages={147--160},
   issn={0025-5874},
}

\end{thebibliography}
\end{document}